\documentclass{article}

\usepackage[a4paper, top=2cm, bottom=2cm, right=1.5cm, left=1.5cm]{geometry} 

\usepackage[T1]{fontenc}               
\usepackage[italian, english]{babel}     
\usepackage{lmodern}                   
\usepackage{textcomp}                  
\usepackage{microtype}                 

\usepackage{mathtools, amssymb, amsthm}

\usepackage{mathbbol}                       

\usepackage{mathrsfs}                  

\usepackage{dsfont}                    

\usepackage{accents}                   

\usepackage{aliascnt}                  

\usepackage{xcolor}                    
\definecolor{myurlcolor}{rgb}{0,0,0.4}
\definecolor{mycitecolor}{rgb}{0,0.5,0}
\definecolor{myrefcolor}{rgb}{0.5,0,0}

\usepackage{graphicx}                  
\usepackage{tikz}                      
\usepackage{tikz-cd}                   
\usetikzlibrary{graphs, decorations.pathmorphing, decorations.markings}
\usepackage{caption}                   

\usepackage{natbib}                    

\usepackage[pagebackref=true, colorlinks=true]{hyperref} 
\hypersetup{
    linkcolor=myrefcolor,
    citecolor=mycitecolor,
    urlcolor=myurlcolor,
}

\usepackage[capitalize, nameinlink]{cleveref}   

\usepackage{changepage}                
\usepackage{enumitem}                  
\usepackage{braket}                    

\newtheorem{theorem}{Theorem}[section]

\newcommand\nuovothm[3]{
  \newaliascnt{#1}{theorem}
  \newtheorem{#1}[#1]{#2}
  \aliascntresetthe{#1}
  \crefname{#1}{#2}{#3}
}

\nuovothm{lemma}{Lemma}{Lemmas}
\nuovothm{proposition}{Proposition}{Propositions}
\nuovothm{corollary}{Corollary}{Corollaries}
\nuovothm{definition}{Definition}{Definitions}
\nuovothm{remark}{Remark}{Remarks}
\nuovothm{assumption}{Assumption}{Assumptions}
\nuovothm{example}{Example}{Examples}
\nuovothm{question}{Question}{Questions}

\newcommand{\be}{\begin{equation}}
\newcommand{\ee}{\end{equation}}

\newcommand{\dd}{{\rm d}}
\newcommand{\de}{\partial}
\DeclareMathOperator\T{\textup{\textbf{T}}}

\title{The coisotropic embedding theorem for pre-symplectic manifolds: an alternative proof}

\author{L. Schiavone$^{1,2}$ \href{https://orcid.org/0000-0002-1817-5752}{\includegraphics[scale=0.7]{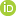}} \\ 
\footnotesize{$^{1}$\textit{Dipartimento di Matematica e Applicazioni Renato Caccioppoli, Università degli Studi di Napoli Federico II}}  \\
\footnotesize{$^{2}$\textit{e-mail: \texttt{luca.schiavone@unina.it}}} 
}

\begin{document}

\maketitle

\begin{abstract}
We present an alternative proof of the Coisotropic Embedding Theorem in which the geometric choice of a connection is recast as the algebraic choice of an embedding into the cotangent bundle. 
The symplectic thickening is then identified as the submanifold determined by the Hamiltonian momenta conjugate to the kernel directions of the pre-symplectic form.
\end{abstract}

\tableofcontents

\section*{Acknowledgements}

I acknowledge financial support from Next Generation EU through the project 2022XZSAFN – PRIN2022 CUP: E53D23005970006.
I am a member of GNSAGA (Indam).

\section*{Introduction}
\label{Sec:Introduction}
\addcontentsline{toc}{section}{\nameref{Sec:Introduction}}

The Coisotropic Embedding Theorem (CET), originally proven by \textit{Mark Gotay} in his seminal work on pre-symplectic manifolds \cite{Gotay-CoisotropicEmbedding-1982}, establishes that any pre-symplectic manifold can be embedded into a symplectic manifold in such a way that the original pre-symplectic manifold becomes a coisotropic submanifold. 
It has significant applications in both theoretical mathematics and applied fields, where it facilitates the study of constrained mechanical systems and reduction theory from various perspectives, as we will recall in this introduction.

\noindent In mathematical physics, Hamiltonian systems are geometrically formulated by means of symplectic manifolds (usually the cotangent bundle of a differential manifold, sometimes referred to as \textit{Phase Space}) endowed with a Hamiltonian function \cite{Arnold-ClassicalMechanics-1989, Abraham-Marsden-Ratiu-FoundationsMechanics-1978}.
On the other hand, constrained Hamiltonian systems are usually described on submanifolds of the Phase Space (selected by the constraints) on which the symplectic form becomes degenerate (see \cite{Gotay-Nester-Hinds-DiracBergmann-1978, Ciaglia-DiCosmo-Ibort-Marmo-Schiavone-Zampini-Peierls1-2024, Ciaglia-DiCosmo-Ibort-Marmo-Schiavone-Zampini-Symmetry-2022} and references therein). 
Among various interesting physical systems, all gauge theories lie in this class of dynamical systems.
The CET ensures that these constrained systems can still be analyzed within the framework of symplectic geometry by embedding the presymplectic manifold into a larger symplectic manifold \cite{Gotay-CoisotropicEmbedding-1982, Abraham-Marsden-Ratiu-FoundationsMechanics-1978, Guillemin-Sternberg-SymplecticTechniques-1990}. 

\noindent Recent advances in mathematical physics have highlighted the utility of the CET in various contexts. 
For instance, in \cite{Ciaglia-DiCosmo-Ibort-Marmo-Schiavone-Zampini-Symmetry-2022}, the theorem is used to establish a Noether-like one-to-one correspondence between symmetries and constants of motion within pre-symplectic Hamiltonian systems.
This correspondence is fundamental in understanding the conservation laws and invariant properties of constrained mechanical systems, thereby extending Noether's classical theorem to a broader context.
Furthermore, the CET has been employed to provide a Poisson structure on the space of solutions of certain field theories. 
For example, in \cite{Ciaglia-DiCosmo-Ibort-Marmo-Schiavone-Zampini-Peierls1-2024, Ciaglia-DiCosmo-Ibort-Marmo-Schiavone-Zampini-Peierls2-2022, Ciaglia-DiCosmo-Ibort-Marmo-Schiavone-Zampini-Peierls3-2022}, the theorem is used to pass from a pre-symplectic framework to a symplectic one, enabling the application of symplectic geometry and Poisson structures to the field theories considered, gauge theories, and Palatini's formulation of General Relativity in particular. 
Moreover, the theorem has been instrumental in addressing the inverse problem of the calculus of variations. 
In recent works \cite{Schiavone-InverseProblemElectrodynamics-2024, Schiavone-InverseProblemImplicit-2024}, the CET has been used to provide a solution for the inverse problem for a class of implicit differential equations. 
These examples illustrate how the theorem can be leveraged to construct variational principles for systems that are not explicitly variational.
Additionally, it is worth mentioning that the CET has significant applications within the Geometric Quantization program \cite{Gotay-Sniatycki-CoisotropicEmbeddings-1981}. 
Geometric Quantization is a procedure for constructing a quantum theory from a classical mechanical system, and requires a symplectic structure for the classical system as starting point. 
By embedding a pre-symplectic manifold into a symplectic one, the theorem facilitates the Quantization of constrained systems by ensuring that the necessary symplectic structure is available.

\noindent Several reformulations of the theorem have been proposed over the years.
We mention the approach of \textit{V. Guillemin} and \textit{S. Sternberg} that proved an equivariant (with respect to the action of a Lie group) version of the theorem \cite{Guillemin-Sternberg-SymplecticTechniques-1990} and the approach of \textit{Y. G. Oh} and \textit{J. S. Park} \cite{Oh-Park-DeformationsCoisotropic-2005} applied to the deformation of coisotropic submanifolds. 

\noindent In this work, we provide an alternative, albeit equivalent, proof of the theorem.
While the classical proof of the CET is well-established, this work aims to present an alternative construction that offers a distinct conceptual and pedagogical viewpoint.
Instead of geometrically constructing a complement to the kernel of the pre-symplectic form, our approach is rooted in the canonical structure of the cotangent bundle $\mathbf{T}^\star M$.
The symplectic thickening is identified by selecting the submanifold of $\mathbf{T}^\star M$ determined by the Hamiltonian functions conjugate to the kernel directions of the pre-symplectic structure.
This perspective recasts the usual arbitrary choice of a \textit{connection}---a purely geometric object---into the choice of an \textit{embedding} into the phase space, a concept more directly tied to its algebraic structure.

\noindent The content of this work is organized as follows. 

\noindent In \cref{Sec:Classical version of the Coisotropic Embedding Theorem}, we recall \textit{Oh} and \textit{Park}'s proof of the CET. 

\noindent In \cref{Sec:A new proof of the Coisotropic Embedding Theorem}, we introduce our alternative proof of the theorem. 

\section{Classical version of the Coisotropic Embedding Theorem}
\label{Sec:Classical version of the Coisotropic Embedding Theorem}

In this section, we recall the classical version of the CET.

\noindent Let us first start with some useful definitions and results.

\begin{definition}[\textsc{Symplectic manifold}] \label{Def: symplectic manifold}
A \textbf{symplectic manifold} is a manifold $M$ equipped with a closed, non-degenerate $2$-form $\omega$, say $(M,\, \omega)$.
\end{definition}

\noindent Here, we recall \textit{Darboux theorem} for symplectic manifolds (see, for instance, \cite{Abraham-Marsden-Ratiu-Manifolds-1988}).

\begin{theorem}[\textsc{Darboux theorem for symplectic manifolds}] \label{Thm: Darboux theorem symplectic}
Given a $2m$-dimensional symplectic manifold $(M,\, \omega)$, there exists a system of local coordinates $$\left\{\, 
x_1,\,...,\,x_m,\,y^1,\,...,\,y^m \,\right\} \,,$$ such that the symplectic form $\omega$ locally reads:
\be
\omega \,=\, \dd x_j \wedge \dd y^j \,.
\ee
\end{theorem}

\begin{definition}[\textsc{Symplectic orthogonal of a submanifold}]
Consider a symplectic manifold $(M,\, \omega)$.
Consider a submanifold $N$ of $M$. 
The \textbf{symplectic orthogonal} of $T_nN$, say ${T_nN}^\perp$ is defined as the set of vectors $v \in T_nM$ such that:
\be
\omega_n(v,\,w) \,=\, 0 \,, \;\;\; \forall \,\, w \in T_nN \,.
\ee
\end{definition}

\begin{definition}[\textsc{Coisotropic submanifold of a symplectic manifold}]
Given a symplectic manifold $(M,\, \omega)$, a submanifold $N$ of $M$ is said to be \textbf{coisotropic} if ${T_nN}^\perp \subseteq T_nM\,,\;\;\; \forall\,\,n \in N$.
\end{definition}

\begin{definition}[\textsc{Pre-symplectic manifold}]
A \textbf{pre-symplectic manifold} is a $n$-dimensional manifold $M$ equipped with closed, non-necessarily non-degenerate, $2$-form $\omega$, say $(M,\, \omega)$.
\end{definition}

\noindent Thus, $\omega$ here may be degenerate.
We will always assume here $\mathrm{rank}\, \omega \,=\, r < n$ to be constant.

\noindent Here, we recall the analog of \cref{Thm: Darboux theorem symplectic} for pre-symplectic manifolds (see, for instance, \cite{Abraham-Marsden-Ratiu-Manifolds-1988}).

\begin{theorem}[\textsc{Darboux theorem for pre-symplectic manifolds}] \label{Thm: Darboux theorem pre-symplectic}
Given a $(2m+r)$-dimensional pre-symplectic manifold $(M,\,\omega)$, there exists a system of local coordinates on $M$ $$\left\{\, x_1,\,...,\,x_m,\,y^1,\,...,\,y^m,\,z^1,\,...,\,z^r \,\right\} \,,$$ such that the pre-symplectic form $\omega$ locally reads:
\be
\omega \,=\, \dd x_j \wedge \dd y^j \,,
\ee
and, thus, $\mathrm{ker}\,\omega$ is spanned, at each point, by:
\be
\mathrm{ker}\,\omega \,=\, \langle \left\{\, \frac{\de}{\de z^1},\,...,\,\frac{\de}{\de z^r} \,\right\}\rangle \,.
\ee
\end{theorem}

\noindent Let us now recall the classical way of proving the CET.

\begin{theorem}[\textsc{Coisotropic embedding theorem}] \label{Thm: coisotropic embedding theorem}
Let $(M, \omega)$ be a pre-symplectic manifold. 
There exists a symplectic manifold $(\tilde{M}, \tilde{\omega})$ such that $(M,\, \omega)$ is a coisotropic submanifold of $(\tilde{M},\, \tilde{\omega})$. 
The symplectic manifold $(\tilde{M},\, \tilde{\omega})$ is called the \textbf{symplectic thickening} of $(M,\, \omega)$.
\end{theorem}

\noindent The proof of the theorem follows three main steps:
\begin{itemize}
    \item \textbf{Construction of the symplectic thickening}. Consider a pre-symplectic manifold $(M, \omega)$. 
    Denote by $K$ the kernel of $\omega$, which is, by assumption, the same for each point of $M$. 
    Thus, it defines a vector subbundle $\mathbf{K}$ of $\T M$, called the \textbf{characteristic bundle}.
    Consider the dual vector bundle $\mathbf{K}^\star$, $\tau$ denoting the canonical projection onto $M$.
    The \textbf{symplectic thickening} of $M$ is (a suitable submanifold of) such a vector bundle.
    \item \textbf{Construction of the symplectic structure}. Consider a complement $W$ of $K$ into $T_mM$, namely a vector subspace of $T_mM$ such that $T_mM \,=\, W \oplus K$.
    It is not unique.
    Indeed, since $K$ is spanned by:
    $$
    \left\{\, \frac{\de}{\de z^1},\,...,\,\frac{\de}{\de z^r} \,\right\} \,,
    $$
    the complement $W$ may be any of the following vector subspaces of $T_m M$:
    \be
    W \,=\, \langle \, \left(\, \frac{\de}{\de x_j} + {P_x}^{ja} \frac{\de}{\de z^a}\,\right)_{j=1,...,m; a=1,...,r} ,\, \left(\,\frac{\de}{\de y^j} + {P_y}_j^a \frac{\de}{\de z^a}\,\right)_{j=1,...,m; a=1,...,r}\, \rangle \,,
    \ee
    where ${P_x}^{ja}$ and ${P_y}_j^a$ are the entries of real-valued rectangular matrices of maximal rank depending on the point $m$.
    The vectors appearing in the latter equation, which are parametrized by the arbitrary values of ${P_x}^{ja}$ and ${P_y}_j^a$, span the kernel of the following system of $1$-forms on $M$:
    \be \label{Eq: connection one-forms}
    P^a \,=\, \dd z^a - {P_x}^{ja} \dd x_j - {P_y}^a_j \dd y^j \,.
    \ee
    Such a system of $1$-forms fulfills the conditions:
    \be
    P^b \left(\, \frac{\de}{\de z^a} \,\right) \,=\, \delta_a^b \,,
    \ee
    which ensures that the $(1,\,1)$-tensor field:
    \be
    P \,:=\, P^a \otimes \frac{\de}{\de z^a} \,,
    \ee
    is idempotent, namely, it is, at each point, a "projector" onto $K$.
    Therefore, the $(1,\,1)$-tensor just constructed is such that, at each point, its image is $K$ and its kernel is the complementary space $W$.
    Such a tensor field is what is usually referred to as a \textit{connection} on $M$, and it is a useful way of parametrizing the arbitrariness in choosing the complement $W$.
    
    \noindent The decomposition $T_mM \,=\, W \oplus K$ provided by $P$ allows us to extend $\omega$ to a symplectic form on $\tilde{M}$.
    Indeed, associated to $P$, there is the following $1$-form on $\mathbf{K}^\star$:
    \be
    \theta^P_p(X_p) \,=\, \langle \,p,\,P\circ T\tau (X_p)\,\rangle\,,
    \ee
    where $X_p \in \mathbf{T}_p \mathbf{K}^\star$, and $p$ has to be interpreted as a point in $\mathbf{K}^\star$ on the left-hand side, and as a covector on $K_m$, with $m = \tau(p)$, on the right-hand side.
    Locally, the latter $1$-form reads
    \be
    \theta^P \,=\, {p_z}_a P^a \,,
    \ee
    where $\left\{\, {p_z}_a \,\right\}_{a=1,...,r}$ is a system of coordinates on the fibers of $\tau$.
    The $1$-form $\theta^P$ evidently depends on the particular connection $P$ chosen, as stressed by the notation.
    It allows one to define the following $2$-form on $\mathbf{K}^\star$
    \be
    \tilde{\omega} \,=\, \tau^\star \omega + \dd \theta^P \,.
    \ee
    It locally reads:
    \be \label{Eq: symplectic structure on the symplectic thickening}
    \tilde{\omega} \,=\, \tau^\star \omega + \dd {p_z}_a \wedge P^a + {p_z}_a \dd P^a \,. 
    \ee
    The $2$-form \eqref{Eq: symplectic structure on the symplectic thickening} is closed by construction, and, as it can be proven via a straightforward computation, is non-degenerate in a tubular neighborhood of the zero-section of $\tau$, which will be denoted as $\tilde{M}$ and represents the symplectic manifold we are searching for.
    As it is shown in \cite{Ciaglia-DiCosmo-Ibort-Marmo-Schiavone-Zampini-Peierls2-2022}, $\tilde{M}$ coincides with the whole $\mathbf{K}^\star$ when the connection $P$ is flat.
    \item \textbf{Check of the coisotropic condition}. 
    The original pre-symplectic manifold $M$ can be naturally embedded as the zero-section of $\tau$ into $\tilde{M}$, namely, it can be recovered as the submanifold of $\tilde{M}$ locally selected by the conditions $${p_z}_a \,=\,0 \,, \;\;\; \forall \,\, a = 1,...,r \,.$$
    A direct computation shows that the symplectic orthogonal of $M$ with respect to the symplectic structure $\tilde{\omega}$ is $\T_mM^\perp = K$. 
    Thus, $\T_mM^\perp \,=\, K \subset T_mM$ is a coisotropic subspace of $\tilde{M}$ by definition.
    \end{itemize}

\section{A new proof of the Coisotropic Embedding Theorem}
\label{Sec:A new proof of the Coisotropic Embedding Theorem}

Here, we propose our alternative proof of the theorem for pre-symplectic manifolds. Our proof will follow the same three steps as the classical one.
\begin{itemize}
    \item \textbf{Construction of the symplectic thickening.}
    Given the pre-symplectic manifold $(M, \omega)$, we begin with its cotangent bundle, $(\T^\star M,\, \omega_{\T^\star M})$, which is a canonical symplectic manifold.  
    Denote by $\rho$ the canonical projection $\rho \;\colon\;\; \mathbf{T}^\star M \to M$.
    Denote by: $$\{\, x_j,\, y^j, \, z^a \,\}_{j=1,...,m; a=1,...,r}$$ ($r$ being again the dimension of $\mathrm{ker}\,\omega$ and $n = 2m + r$ being the dimension of $M$) a system of Darboux coordinates on $M$ and by: $$\{\, x_j,\, y^j,\, z^a,\, {p_x}^j,\, {p_y}_j,\, {p_z}_a \,\}_{j=1,...,m; a=1,...,r}$$ an adapted system of coordinates on $\mathbf{T}^\star M$.
    In this system of coordinates 
    \be
    \omega \,=\, \dd x_j \wedge \dd y^j \,,
    \ee 
    and 
    \be
    \omega_{\mathbf{T}^\star M} \,=\, \dd {p_x}^j \wedge \dd x_j + \dd {p_y}_j \wedge \dd y^j + \dd {p_z}_a \wedge \dd z^a \,.
    \ee
    The form $\omega' \,=\, \omega_{\mathbf{T}^\star M} + \rho^\star \omega$ is symplectic as well.
    Indeed, to prove that it is non-degenerate it is sufficient to prove that $$\omega'^n \,=\, \omega' \underbrace{\,\wedge\, ...\, \wedge\,}_{n\, \text{times}} \omega'$$ is different from zero for $n$ equal to half the dimension of $\mathbf{T}^\star M$. 
    A direct computation shows that $\omega'^n \,=\, \omega_{\mathbf{T}^\star M}^n$, where $\omega_{\mathbf{T}^\star M}^n$ is different from zero because $\omega_{\mathbf{T}^\star M}$ is symplectic.
    
    \noindent In the system of coordinates chosen, $\omega'$ reads
    \be
    \omega' \,=\, \dd {p_x}^j \wedge \dd x_j + \dd {p_y}_j \wedge \dd y^j + \dd {p_z}_a \wedge \dd z^a + \dd x_j \wedge \dd y^j \,.
    \ee
    Recall that $\left\{\, \frac{\de}{\de z^a} \,\right\}_{a=1,...,r}$ is a system of generators of $\mathrm{ker}\, \omega \,=\, K$.
    Their cotangent lifts are those vector fields $\tilde{K}_a$ tangent to $\mathbf{T}^\star M$ defined by the condition
    \be
    \mathcal{L}_{\tilde{K}_a} \omega_{\mathbf{T}^\star M} \,=\, 0 \,,
    \ee
    where $\mathcal{L}_{\tilde{K}_a}$ denotes the Lie derivative along the vector field $\tilde{K}_a$.
    They are computed to be
    \be
    \tilde{K}_a \,=\, \frac{\de}{\de z^a} \,.
    \ee
    Denote by $\left\{\, H_a \,\right\}_{a=1,...,r}$ the Hamiltonian functions associated to $\tilde{K}_a$ via $\omega'$, namely those functions satisfying
    \be
    i_{\tilde{K}_a} \omega' \,=\, \dd H_a \,,
    \ee
    where $i_{\tilde{K}_a}$ denotes the contraction of the differential form along the vector field $\tilde{K}_a$.
    They read
    \be
    H_a \,=\, {p_z}_a \,.
    \ee

    \noindent The Hamiltonian functions $H_a$ are linear on the fibers of the cotangent bundle and represent the momenta conjugate to the kernel directions $z_a$. 
    
    \noindent We now construct the thickening manifold $\tilde{M}$ by applying the Gel'fand-Kolmogorov theorem \cite{Marmo-Morandi-SomeGeometry-1994}.
    \begin{proposition}
    The bundle $\mathbf{K}^\star$ of the classical proof of the CET is the manifold associated, by means of the Gel'fand-Kolmogorov theorem, with the algebra $\mathcal{A} \subset C^\infty(\T^\star M)$ generated by functions pulled back from $M$ and the kernel-momenta functions $\{H_a\}$. 
    \end{proposition}
    \begin{proof}
    A point in the manifold associated with the algebra $\mathcal{A}$ can be identified with a character $\chi\,:\; \mathcal{A} \to \mathbb{R}$, which is an algebra homomorphism. 
    Any such character is determined by its action on the generators of $\mathcal{A}$.
    \begin{enumerate}
        \item The restriction of $\chi$ to the subalgebra of functions pulled back from the base, $\{f \circ \rho \;:\;\; f \in C^\infty(M)\}$, is a character on $C^\infty(M)$. 
        The space of characters of $C^\infty(M)$ is the manifold $M$ itself. 
        Therefore, for each $\chi$, there exists a unique point $m \in M$ such that $\chi(f \circ \rho) = f(m)$ for all $f \in C^\infty(M)$. 
        This fixes the base point of our new manifold.
        \item The action of $\chi$ on each Hamiltonian generator $H_a$ yields a real number, $p_a := \chi(H_a)$. 
        This set of numbers $\{p_a\}_{a=1}^r$ can be used to define a covector on the kernel fiber $K_m$.
        Specifically, we define $\alpha_m \in K_m^\star$ as the linear functional on $K_m$ such that $\braket{\alpha_m,\, \frac{\de}{\de z^a}} = p_a$.
    \end{enumerate}
    Thus, each character $\chi$ corresponds to a unique pair $(m, \alpha_m)$ where $m \in M$ and $\alpha_m \in K_m^\star$. 
    This establishes a bijection between the space of characters of $\mathcal{A}$ and the total space of the dual bundle $K^\star$. 
    Since $K^\star$ is a vector bundle over a smooth manifold, it is itself a smooth manifold. 
    \end{proof}

    \noindent There are many, non-equivalent ways of embedding $\mathbf{K}^\star$ into $\mathbf{T}^\star M$ as a vector subbundle.
    This class of embeddings is of the type
    \be
    \begin{split}
    \mathfrak i \;\colon \;\; \mathbf{K}^\star \to \mathbf{T}^\star M \;\;:\;\; &(\,x_j,\,y^j,\, z^a,\, {p_z}_a \,) \\ 
    \mapsto &\left(\, x_j,\, y^j,\, z^a,\, {p_x}^j \,=\, -{P_x}^{ja} {p_z}_a,\, {p_y}_j \,=\, -{P_y}^a_j {p_z}_a ,\, {p_z}_a\,\right) \,,
    \end{split}
    \ee
    where ${P_x}^{ja}$ and ${P_y}^a_j$ are the entries of two rectangular matrices of maximal rank depending on the base coordinates $(x_j,\,y^j,\, z^a)$, and where the minus sign is an arbitrary choice.
    The fact that $\mathfrak{i}$ is linear in the ${p_z}_a$ ensures that $\mathfrak{i}$ preserves the zero-section of $\rho$.
    On the other hand, the maximal rank hypothesis ensures that the fibres of $\rho\bigr|_{Im \, \mathfrak{i}}$ are all diffeomorphic, and, thus, the image of $\mathfrak{i}$ defines a vector subbundle of $\mathbf{T}^\star M$.

    \noindent It is crucial to note that the functions ${P_x}^{ja}$ and ${P_y}^a_j$ that define the embedding $\mathfrak{i}$ play precisely the same role as the components of the connection $1$-forms \eqref{Eq: connection one-forms} in the classical proof. 
    Our "choice of embedding" is therefore the direct counterpart of the classical "choice of a connection".

    \item \textbf{Construction of the symplectic structure.}
    The symplectic structure is constructed by taking the pullback of the ambient symplectic form $\omega'$ along the embedding $\mathfrak{i}$:
    \begin{equation}
        \tilde{\omega} := \mathfrak{i}^\star \omega' = \mathfrak{i}^\star (\omega_{\T^\star M} + \rho^\star \omega) = \mathfrak{i}^\star \omega_{\T^\star M} + \mathfrak{i}^\star \rho^\star \omega\,,
    \end{equation}
    which locally reads
    \begin{equation}
    \tilde{\omega} = \rho_{K^\star}^\star \omega + \dd p_a \wedge P^a + p_a \dd P^a\,,
    \end{equation}
    which is identical to the form \eqref{Eq: symplectic structure on the symplectic thickening} derived in the classical construction.
    Thus, as in the classical case, one can prove that $\tilde{\omega}$ is symplectic in a tubular neighborhood of the zero-section of the bundle $\tau \;:\;\; \mathbf{K}^\star \to M$, and that such a zero-section can be extended to the whole bundle if the connection $P$ is flat.
    Again, we will denote the tubular neighborhood above as $\tilde{M}$, and it represents the symplectic manifold we are searching for.
    \item \textbf{Check of the coisotropic condition.}
    Since the symplectic form $\tilde{\omega}$ we have constructed is identical to the one from the classical proof, the calculation of the symplectic orthogonal of $M$ is also identical.
    A direct computation confirms that for any point $m \in M$, the symplectic orthogonal of the tangent space to the embedded manifold is precisely the kernel $K_m$.
    Consequently, since $K_m \subseteq T_m M$, one concludes that $M$ is a coisotropic submanifold of $(\tilde{M}, \tilde{\omega})$.
\end{itemize}

\section*{Conclusions}
\label{Sec:Conclusions}
\addcontentsline{toc}{section}{\nameref{Sec:Conclusions}}

We provided an alternative proof of the CET for pre-symplectic manifolds that passes through the cotangent bundle of the pre-symplectic manifold and identifies the symplectic thickening as the submanifold of such a cotangent bundle selected by the Hamiltonian functions associated with the vectors in the kernel of the original pre-symplectic form.
The classical ambiguity in choosing a complement for the kernel of the pre-symplectic structure is recovered as the ambiguity in embedding such a submanifold into the cotangent bundle over the original pre-symplectic manifold.

\bibliographystyle{alpha}
\bibliography{Biblio}

\newcommand{\etalchar}[1]{$^{#1}$}
\begin{thebibliography}{CDI{\etalchar{+}}22b}

\bibitem[AMR88]{Abraham-Marsden-Ratiu-Manifolds-1988}
R.~Abraham, J.~E. Marsden, and T.~Ratiu.
\newblock {\em {M}anifolds, {T}ensor {A}nalysis, and {A}pplications}.
\newblock Springer, New York, 1988.

\bibitem[AMRC78]{Abraham-Marsden-Ratiu-FoundationsMechanics-1978}
R.~Abraham, J.~E. Marsden, T.~Ratiu, and R.~Cushman.
\newblock {\em {F}oundations of {M}echanics}.
\newblock Addison-Wesley Publishing Company, Inc., Redwood City, 1978.

\bibitem[Arn89]{Arnold-ClassicalMechanics-1989}
V.~I. Arnold.
\newblock {\em {M}athematical methods of {C}lassical {M}echanics}.
\newblock Springer-Verlag, New York, 1989.

\bibitem[CDI{\etalchar{+}}22a]{Ciaglia-DiCosmo-Ibort-Marmo-Schiavone-Zampini-Symmetry-2022}
F.~M. Ciaglia, F.~{Di Cosmo}, L.~A. Ibort, G.~Marmo, L.~Schiavone, and A.~Zampini.
\newblock {S}ymmetries and {C}ovariant {P}oisson {B}rackets on {P}resymplectic {M}anifolds.
\newblock {\em Symmetry}, 14(70):1--28, 2022.

\bibitem[CDI{\etalchar{+}}22b]{Ciaglia-DiCosmo-Ibort-Marmo-Schiavone-Zampini-Peierls2-2022}
F.~M. Ciaglia, F.~{Di Cosmo}, L.~A. Ibort, G.~Marmo, L.~Schiavone, and A.~Zampini.
\newblock {T}he geometry of the solution space of first order {H}amiltonian field theories {II}: non-{A}belian gauge theories, 2022.

\bibitem[CDI{\etalchar{+}}23]{Ciaglia-DiCosmo-Ibort-Marmo-Schiavone-Zampini-Peierls3-2022}
F.~M. Ciaglia, F.~{Di Cosmo}, L.~A. Ibort, G.~Marmo, L.~Schiavone, and A.~Zampini.
\newblock {T}he geometry of the solution space of first order {H}amiltonian field theories {III}: {P}alatini's formulation of {G}eneral {R}elativity, 2023.

\bibitem[CDI{\etalchar{+}}24]{Ciaglia-DiCosmo-Ibort-Marmo-Schiavone-Zampini-Peierls1-2024}
F.~M. Ciaglia, F.~{Di Cosmo}, L.~A. Ibort, G.~Marmo, L.~Schiavone, and A.~Zampini.
\newblock {T}he geometry of the solution space of first order {H}amiltonian field theories {I}: from particle dynamics to {E}lectrodynamics.
\newblock {\em Journal of Geometry and Physics}, 204:105279, 2024.

\bibitem[GNH78]{Gotay-Nester-Hinds-DiracBergmann-1978}
M.~J. Gotay, J.~M. Nester, and G.~Hinds.
\newblock {P}resymplectic manifolds and the {D}irac–{B}ergmann theory of constraints.
\newblock {\em Journal of Mathematical Physics}, 19(11):2388--2399, 1978.

\bibitem[Got82]{Gotay-CoisotropicEmbedding-1982}
M.~J. Gotay.
\newblock {O}n coisotropic imbeddings of presymplectic manifolds.
\newblock {\em Proceedings of the American Mathematical Society}, 84(1):111--114, 1982.

\bibitem[GS81]{Gotay-Sniatycki-CoisotropicEmbeddings-1981}
M.~J. Gotay and J.~Sniatycki.
\newblock {O}n the quantization of presymplectic dynamical systems via coisotropic imbeddings.
\newblock {\em Communications in Mathematical Physics}, 82(3):377--389, 1981.

\bibitem[GS90]{Guillemin-Sternberg-SymplecticTechniques-1990}
V.~Guillemin and S.~Sternberg.
\newblock {\em {S}ymplectic techniques in physics}.
\newblock Cambridge University Press, 1990.

\bibitem[MM94]{Marmo-Morandi-SomeGeometry-1994}
G.~Marmo and G.~Morandi.
\newblock Some {G}eometry and {T}opology.
\newblock In S.~Lundqvist, G.~Morandi, and Yu~Lu, editors, {\em Low-{D}imensional {Q}uantum {F}ield {T}heories for {C}ondensed {M}atter {P}hysicists}, Modern Condensed Matter Physics, pages 1--107. World-Scientific, 1994.

\bibitem[OP05]{Oh-Park-DeformationsCoisotropic-2005}
Y.~G. Oh and J.~S. Park.
\newblock {D}eformations of coisotropic submanifolds and strong homotopy {L}ie algebroids.
\newblock {\em Inventiones Mathematicae}, 161:287--360, 2005.

\bibitem[Sch24a]{Schiavone-InverseProblemImplicit-2024}
L.~Schiavone.
\newblock {T}he inverse problem for a class of implicit differential equations and the coisotropic embedding theorem.
\newblock {\em International Journal of Geometric Methods in Modern Physics}, 21(11):2450195, 2024.

\bibitem[Sch24b]{Schiavone-InverseProblemElectrodynamics-2024}
L.~Schiavone.
\newblock {T}he inverse problem within free {E}lectrodynamics and the coisotropic embedding theorem.
\newblock {\em International Journal of Geometric Methods in Modern Physics}, 21(7):2450131, 2024.

\end{thebibliography}

\end{document}